\documentclass[11pt,reqno]{amsart}

\usepackage[T1]{fontenc}
\usepackage[english]{babel}

\usepackage[margin=1.2in,footskip=.2in]{geometry}\usepackage{amsmath,amssymb,amscd,amsfonts,mathtools,tikz,caption,float}
\usetikzlibrary{patterns}
\usepackage[linkcolor=blue,colorlinks=true,citecolor=blue]{hyperref}
\usepackage[capitalise]{cleveref}

\newtheorem{thm}{Theorem}[section]

\numberwithin{equation}{section}

\newcommand{\abs}[1]{\left\vert#1\right\vert}

\newcommand{\R}{\mathbb{R}}
\newcommand{\C}{\mathbb{C}}
\newcommand{\T}{\mathbb{T}}

\newcommand{\h}{\mathbb{H}}
\newcommand{\Q}{\mathbb{Q}}
\newcommand{\Z}{\mathbb{Z}}

\newcommand{\eps}{\varepsilon}

\newcommand{\im}{\mathfrak{Im}}
\newcommand{\re}{\mathfrak{Re}}

\newcommand{\bbm}{\begin{bmatrix}}
\newcommand{\ebm}{\end{bmatrix}}
\newcommand{\bpm}{\begin{pmatrix}}
\newcommand{\epm}{\end{pmatrix}}
\newcommand{\bsm}{\left(\begin{smallmatrix}}
\newcommand{\esm}{\end{smallmatrix}\right)}
\newcommand{\bsbm}{\left[\begin{smallmatrix}}
\newcommand{\esbm}{\end{smallmatrix}\right]}

\newcommand{\sign}{\mathrm{sign}}

\newcommand{\1}{\mathrm{1}}
\newcommand{\SL}{\mathrm{SL}}

\begin{document}

\title{Dedekind sums, reciprocity, and non-arithmetic groups}
\author{Claire Burrin}
\address{Dept.~of Mathematics, Rutgers University, 110 Frelinghuysen Rd, Piscataway, NJ 08854}
\email{claire.burrin@rutgers.edu}
\date{\today}
\maketitle

Dedekind sums, arithmetic correlation sums that arose in Dedekind's study of the modular transformation of the logarithm of the $\eta$-function \cite{Ded}, are surprisingly ubiquitous. Their arithmetic properties attracted the attention of number theorists, combinatorists, and theoretical computer scientists alike \cite{Rad,Mey,Pom,Knu,BR}, and they appear more broadly in geometry, topology, and physics \cite{Ati,KM,BG}. Accordingly, there is a similarly vast literature on variations and generalizations of Dedekind sums. It is the goal of this note to survey some of the aspects of Dedekind sums for (non-uniform) lattices in $\SL_2(\R)$, otherwise referred to as Dedekind symbols. Intrinsically, this gives us a framework in which to investigate the r\^ole `arithmeticity' plays in defining properties of Dedekind sums. In this note, we discuss the reciprocity law for Dedekind symbols associated to lattices that are not necessarily arithmetic. We will see that the reciprocity law holds given an algebraic structure similar to that of $\SL_2(\Z)$, e.g.~Hecke triangle groups.

\section{Dedekind sums}

Let $x\mapsto((x))$ be the function that assigns to any real number $x$ the value
$$
((x)) = \begin{cases}  0 & x\in\Z\\ x-\lfloor x\rfloor -1/2 & x\in\R\setminus\Z, \end{cases}
$$
For any pair of coprime integers $(a,c)$ with $c>0$, the Dedekind sum is 
\begin{align}\label{Dedekind sum}
s(a,c) = \sum_{k=1}^{c-1} \left(\!\left(\frac{k}{c}\right)\!\right)\left(\!\left(\frac{ak}{c}\right)\!\right).
\end{align}
An intuitive way of thinking of Dedekind sums is perhaps the following. Consider the sequence $(y_n)=\left(\frac{x_n}{c}\right)$ of primitive rational points given iteratively by the linear congruential rule
\begin{align*}
x_0 &= 1\\
x_{n} &\equiv ax_{n-1} \text{ mod }c
\end{align*}
defined by a pair of coprime integers $(a,c)$. By Euler's theorem, $(y_n)$ is a periodic sequence, whose period is completely determined by the modulus $c$. More precisely, the period is given by Euler's totient function $\varphi(c)$, and hence is maximal if the modulus $c$ is a prime. Moreover, in this case, $(y_n)$ runs over all primitive rational points with denominator $c$. Then the Dedekind sum $s(a,c)$ is the covariance of successive points in $(y_n)$. 
In other words, a large Dedekind sum expresses high correlation between successive points, which 
suggests recurring patterns in the distribution of these points in the unit interval, while a small Dedekind sum indicates the possible asymptotic uniformly distribution of these points in the unit interval if we let $c\to\infty$. If such is the case, $(y_n)$ is called pseudo-random, meaning that although it is the result of a deterministic process, the output values look random. 
In short, Dedekind sums measure the pseudorandomness of $(y_n)$ depending on the input pair $(a,c)$ of coprime integers.

\section{...and modular transformations}

As a guiding introduction to the construction of Dedekind symbols, we will see where Dedekind sums appear in analytic number theory. From this viewpoint, where best to start than with the Riemann $\zeta$-function! Recall that 
$$
\zeta(s) = \sum_{n\geq1} \frac{1}{n^s}
$$ 
is defined for $\re(s)>1$, has a meromorphic continuation to the whole complex plane that is regular except for a simple pole at $s=1$ with residue 1, and the constant term in the Laurent expansion of $\zeta(s)$ at 1 is the Euler--Mascheroni constant $C$, i.e.
$$
\lim_{s\to1} \left(\zeta(s)-\frac{1}{s-1}\right) = C.
$$
More generally, any real positive-definite binary quadratic form  $Q(x,y)=ax^2 +bxy +cy^2$, defines a $\zeta$-function
$$
\zeta_Q(s) = \sum_{m,n\in\Z} \frac{1}{Q(m,n)^s} = \sum_{\substack{m,n\in\Z\\\text{not both }0}} \frac{y^s}{\abs{m+nz}^{2s}} \qquad \left(\begin{array}{rcl} z &= &x+iy\\ &= &\frac{1}{a}\left(b+\sqrt{-d}\right)\end{array}\right)
$$
for $\re(s)>1$. The first limit formula of Kronecker asserts that
$$
\lim_{s\to1}\left(\zeta_Q(s)-\frac{\pi}{s-1}\right) = 2\pi\left( C-\log2-\log\sqrt{y}\abs{\eta(z)}^2\right)
$$
where $\eta(z)$ is Dedekind's $\eta$-function
$$
\eta(z) = e^{\frac{\pi iz}{12}} \prod_{n\geq1} (1-e^{2\pi inz}).
$$
From Kronecker's first limit formula, one can show that for any $\gamma\in\SL_2(\Z)$, 
$$
\eta(\gamma z) = \eps(cz+d)^{1/2}\eta(z)
$$
with $\eps=\eps(a,b,c,d)$ and $\abs{\eps}=1$ (see for instance \cite{Sie}). Dedekind \cite{Ded} determined $\eps$ explicitly, by showing that
\begin{align}\label{Ded}
\log\eta(\gamma z) = i\pi\phi(\gamma) +\frac{1}{2}\log\left(\frac{cz+d}{i\ \sign(c)}\right) +\log\eta(z)
\end{align}
where $\log$ denotes the principal branch of the logarithm, $\sign(c)=\frac{c}{\abs{c}}$ and 
$$
\phi\bpm a&b\\ c&d\epm = \begin{dcases}
\frac{b}{d} & c=0\\
\frac{1}{12}\frac{a+d}{c} -\sign(c)s(a,\abs{c}) & c\neq0.
\end{dcases}
$$
The modular transformation (\ref{Ded}) of $\log\eta$ can be exploited to say things about Dedekind sums. A famous example is Dedekind's reciprocity law
\begin{align}\label{DRL}
s(a,c)+s(c,a) = \frac{1}{12}\left(\frac{a}{c}+\frac{1}{ac}+\frac{c}{a}\right) -\frac{1}{4}
\end{align}
for $a,c$ positive and coprime.

\section{Dedekind symbols}

Let $\Gamma\subset\SL_2(\R)$ be a lattice (i.e.~a discrete subgroup of finite covolume). Then $\Gamma$ acts on the (hyperbolic) upper half-plane $\h=\{x+iy:y>0\}$ properly discontinuously by fractional linear transformation, $\bsm a&b\\ c&d\esm.z = \frac{az+b}{cz+d}$, and this action extends to $\overline{\h}=\h\cup\partial\h=\h\cup\R\cup\{\infty\}$. For simplicity, we will first assume that $\Gamma$ has only one cusp (i.e.~the action of $\Gamma$ has a single fixed point in $\partial\h$) and that this cusp is the point at $\infty$. Let $\Gamma_\infty\subset\Gamma$ be the stabilizer subgroup of the cusp. Up to normalization, we may assume that $\Gamma_\infty=\bsm 1&\Z\\&1\esm$. The Eisenstein series
$$
E(z,s) = \sum_{\gamma\in\Gamma_\infty\backslash\Gamma} \im(\gamma z)^s = \sum_{\gamma\in\Gamma_\infty\backslash\Gamma} \frac{y^s}{|cz+d|^{2s}}
$$
is defined for $\re(s)>1$, and has a meromorphic continuation to the whole complex plane that is holomorphic in $\re(s)\geq1$ except for a simple pole at $s=1$ with residue $V^{-1}$, where $V$ is the covolume of $\Gamma$. 
Goldstein \cite{Go} derived (formally) the Kronecker first limit formula
$$
\lim_{s\to1}\left( E(z,s) - \frac{V^{-1}}{s-1}\right) = K(z)
$$
to obtain Dedekind sums for principal congruence subgroups. The existence of this limit can be deduced from bounds on the Fourier coefficients of the Eisenstein series, see \cite{JO}. We present a softer approach, using only that the Eisenstein series are eigenfunctions of the Laplacian on $\h$, expressed in local coordinates $z=x+iy$ by $\Delta=-y^2(\partial_{xx}+\partial_{yy})$. In fact, 
$$
\Delta E(z,s) = s(1-s)E(z,s)
$$
and note that $\Delta$ and the $\Gamma$-action commute, i.e. $\Delta f(\gamma z)=(\Delta f)(\gamma z)$, and $E(\gamma z,s)=E(z,s)$ for all $\gamma\in\Gamma$. We can observe that
\begin{enumerate}
\item $K(\gamma z)=K(z)$ for all $\gamma\in\Gamma$,
\item $K(z)$ is real-valued and real-analytic,
\item $\Delta K(z) = -V^{-1}$. 
\end{enumerate}
The last two facts imply that
\begin{align*}
H(z) = V K(z) +\ln y +c,
\end{align*}
for any fixed real constant $c$, is harmonic, i.e.~$\Delta H(z)=0$, and that it can be realized as the real part of a holomorphic function $F:\h\to\C$, i.e.~$\re(F(z))=H(z)$. Hence 
\begin{align}\label{klf}
\lim_{s\to1}\left( E(z,s) - \frac{V^{-1}}{s-1}\right) = V^{-1}\left(c - \ln y + \re F(z)\right).
\end{align}
(The constant $c$ can be made explicit by looking at the Fourier coefficients of the Eisenstein series.) 

To obtain Dedekind symbols, we study the automorphic transformation of the holomorphic function $F$. The Kronecker limit formula (\ref{klf}) and Fact (1) imply
\begin{align*}
 \re\left( F(z) -F(\gamma z)\right) = \ln\abs{cz+d}^2 = \re\left(\log(-(cz+d)^2)\right),
 \end{align*}
 where we fix $\log$ to be the principal branch of the logarithm. As a result, for any fixed group element $\gamma\in\Gamma$, the function
$$
\phi_\gamma(z) = F(z) - F(\gamma z) - \log(-(cz+d)^2),
$$
is holomorphic and takes values on the imaginary axis. Therefore, by the Open Mapping Theorem, it is constant, i.e.~$\phi(\gamma)\equiv\phi_\gamma(z)$. Since what we are really interested in is the variation in the argument, we replace $\phi$ by the normalization $\frac{\phi}{2\pi i}$, so that
\begin{align}\label{phi}
\phi(\gamma) = \frac{1}{2\pi i}\left( F(z)-F(\gamma z) -\log(-(cz+d)^2)\right).
\end{align} 
By analogy with Dedekind's formula (\ref{Ded}) for the transformation of $\log\eta$, we define for any $\bsm a&b\\c&d\esm\in\Gamma$ such that $c\neq0$,
$$
\mathcal{S}(a,b,c,d) = \frac{V}{4\pi} \frac{a+d}{c} - \phi\bpm a&b\\ c&d\epm.
$$
The right-hand side (RHS) of this equation factors through $\Gamma_\infty\backslash\Gamma\slash\Gamma_\infty$ \cite[Theorem 2]{me}. Observe that for $\gamma=\bsm a&b\\c&d\esm$ and $\gamma'=\bsm a&b'\\c&d'\esm$ in $\Gamma$,
$$
\gamma^{-1}\gamma' =  \bpm d&-b\\-c&a\epm \bpm a&b'\\ c&d'\epm = \bpm 1& *\\ 0&1\epm \in\Gamma_\infty
$$
meaning that any double coset representative $[\![\bsm a&b\\c&d\esm]\!] = \Gamma_\infty\bsm a&b\\c&d\esm\Gamma_\infty$ is completely determined by the column vector $(a,c)^T$. The same argument can be used to show that the double coset representative $[\![\bsm a&b\\c&d\esm]\!]$ can as well be completely determined by the row vector $(c, d)$. Hence, we have the Dedekind symbol
$$
\mathcal{S}_c(a,c) = \mathcal{S}_r(c,d) = \mathcal{S}\left[\!\left[\bpm a&b\\c&d\epm\right]\!\right] =
\frac{V}{4\pi} \frac{a+d}{c} - \phi\bpm a&b\\ c&d\epm
$$
for $\Gamma$. More generally, if $\Gamma\subset \SL_2(\R)$ has cusps, then each cusp $\frak{a}$ gives rise to a Dedekind symbol
\begin{align}\label{symbol}
\mathcal{S}_\frak{a} ([\![\gamma]\!]) = \frac{V}{4\pi} \frac{a_\gamma+d_\gamma}{c_\gamma} - \phi\bpm a_\gamma &b_\gamma\\ c_\gamma & d_\gamma\epm
\end{align}
for each non-trivial double coset $[\![\gamma]\!]$ in $\Gamma_\frak{a}\backslash\Gamma\slash\Gamma_\frak{a}$ and
$$
\bpm a_\gamma&b_\gamma\\ c_\gamma&d_\gamma\epm = \sigma_\frak{a}^{-1}\gamma\sigma_\frak{a},
$$ 
where $\sigma_\frak{a}\in\SL_2(\R)$ is a scaling, i.e.~$\sigma_\frak{a}(\infty)=\frak{a}$ and $\sigma_\frak{a}^{-1}\Gamma_\frak{a}\sigma_\frak{a}=\bsm 1&\Z\\&1\esm$. The definition (\ref{symbol}) does not depend on the particular choice of $\sigma_\frak{a}$.

\section{Equidistribution mod 1}

The perhaps {\em a priori} artificial definition (\ref{symbol}) of Dedekind symbols above is justified by the following theorem (which is the effective version of \cite[Theorem 4]{me} and can be compared to \cite{Va,Va2}.)

\begin{thm}
Let $\Gamma\subset\SL_2(\R)$ be a non-uniform lattice, with a cusp at $\frak{a}$ and associated Dedekind symbol $\mathcal{S}_\frak{a}$. The values of 
$$
\mathcal{S}_\frak{a}([\![\gamma]\!]) = \mathcal{S}_c(a_\gamma,c_\gamma),
$$
running over ordered admissible pairs $(a_\gamma,c_\gamma)$ as in (\ref{symbol}), become equidistributed mod 1, as $c_\gamma\to\infty$. More precisely, let $\mu_X$ be the normalized counting measure defined on the subset
$$
\mathcal{D}_X = \left\{ \mathcal{S}_c(a_\gamma,c_\gamma)\text{ mod } 1: 0\leq a_\gamma<c_\gamma\leq X\right\}
$$ 
of $\T=\R\slash\Z$, and let $\mu$ denote the Lebesgue measure on $\T$. Then for any $f\in C^\infty(\T)$,
$$
\mu_X(f) = \mu(f) + O\left(X^{-\frac23 +\eps}\right)
$$
as $X\to\infty$. The implied constant depends on a Sobolev norm of $f$ and $\Gamma$.
\end{thm}

\begin{proof}
Any $f\in C^\infty(\T)$ admits a Fourier series 
$$
f(x) = \sum_{n\in\Z} a_n e(nx) := \sum_{n\in\Z} a_n e^{2\pi inx}
$$
that converges uniformly. Then
\begin{align*}
\mu_X(f) &= \frac{1}{\abs{\mathcal{D}_X}} \sum_{x\in \mathcal{D}_X}  f(x) \\
& =  \frac{1}{\abs{\mathcal{D}_X}}\sum_{n\in\Z} a_n\sum_{\substack{a_\gamma<c_\gamma\\c_\gamma\leq X}} e(n\mathcal{S}_c(a_\gamma,c_\gamma))\\
 &= \frac{1}{\abs{\mathcal{D}_X}} \sum_{n\in\Z}a_n \sum_{\substack{a_\gamma<c_\gamma\\c_\gamma\leq X}} e\left(n\frac{V}{4\pi} \frac{a+d}{c}\right) e\left(-n\phi(\gamma)\right)\\
& = \frac{1}{\abs{\mathcal{D}_X}}\sum_{n\in\Z} a_n  \sum_{c_\gamma\leq X}S(n,n,c_\gamma,\phi)
\end{align*}
is expressed in function of sums of $\phi$-twisted Selberg--Kloosterman sums \cite{Sel}. The constant term in this expansion is 
$$
\frac{a_0}{\abs{\mathcal{D}_X}}\sum_{c_\gamma\leq X} S(0,0,c_\gamma,\phi)= a_0 = \int_\T f(x) dx = \mu(f).
$$
Subconvex bounds on the sums of $\phi$-twisted Selberg--Kloosterman sums sums not only yield $\mu_X(f)\to \mu(f)$ as $X\to\infty$, they also control the rate of equidistribution, as we will now see. Thanks to the spectral theory of automorphic forms, subconvex bounds for (general) sums of Kloostermans sums are available, and in the specific case of the $\phi$-twisted Selberg--Kloosterman sums considered here, \cite[Theorem 2]{GS} together with \cite[Satz 5.5]{Roe}, yield, for every pair $m,n\in\Z$ such that $mn\neq0$,
$$
\sum_{c_\gamma\leq X} S(m,n,c_\gamma,\phi) \ll \abs{mn}X^{\frac{4}{3}+\eps},
$$
$\eps>0$, with an implied constant depending on $\Gamma$. 

In \cite{me}, we showed that 
$$
\abs{\mathcal{D}_X}\sim \frac{V}{4\pi}X^2
$$
as $X\to\infty$. Combining these results yields
\begin{align*}
\abs{\mu_X(f) - \mu(f)} \ll \sum_{n\neq0} \abs{a_n} \abs{n}^2 X^{-\frac23+\eps} \ll  \left( \sum_{n\neq0} \abs{a_n}^2 \abs{n}^8\right)^{\frac12} X^{-\frac23+\eps} \leq S(f) X^{-\frac23+\eps},
\end{align*}
where we used Cauchy--Schwarz in the second step, $S(f)$ is the adequate $L^2$-Sobolev norm, and the implied constant in the first step depends on $\Gamma$.
\end{proof}

\section{Reciprocity of Dedekind symbols}

We will now see that there is an analogue of the reciprocity law (\ref{DRL}) for Dedekind symbols. We extract from the construction of Dedekind symbols the formula
$$
F(z) - F(\gamma z) = \log(-(cz+d)^2) +2\pi i\phi(\gamma) = 2\log(cz+d) + 2\pi i\psi(\gamma)
$$
with
$$
\psi(\gamma) = \phi(\gamma)-\frac{1}{4}\sign(c),
$$
for which one can check directly that
\begin{align}\label{cocycle}
\psi(\gamma\tau)-\psi(\gamma)-\psi(\tau) = \frac{1}{2\pi i}\left( \log j(\gamma,\tau z) +\log j(\tau,z) - \log j(\gamma\tau,z)\right),
\end{align}
where $j(g,z)=cz+d$, $g=\bsm *&*\\ c&d\esm$.
The RHS is independent of $z$ and extends naturally to a function on $\SL_2(\R)\times\SL_2(\R)$. Moreover, it can be computed on explicit elements of $\SL_2(\R)$. In fact, in \cite[Theorem 1.2]{me2}, we show that,
 for any $g=\bsm a_g & b_g\\ c_g& d_g\esm$, $h=\bsm a_h& b_h\\ c_h &d_h\esm$ in $\SL_2(\R)$, and any $z\in\h$,
\begin{align*}
\omega & (g,h) := \frac{1}{2\pi i}\left( \log j(g,h z) +\log j(h,z) - \log j(gh,z)\right)\\
&= \frac14\left\{ \sign(c_g(-d_g))+\sign(c_h(-d_h))-\sign(c_{gh}(-d_{gh})) - \sign(c_g(-d_g)c_h(-d_h)c_{gh}(-d_{gh}))\right\}
\end{align*}
where
$$
c(-d) = \begin{dcases} c & c\neq0,\\ -d & c=0.\end{dcases}
$$
Let $\gamma, \tau\in\Gamma$ such that $[\![\gamma]\!]$, $[\![\tau]\!]$, $[\![\gamma\tau]\!]$ are non-trivial double cosets in $\Gamma_\frak{a}\backslash\Gamma\slash\Gamma_\frak{a}$. As before, write $\bsm a_\gamma & b_\gamma\\ c_\gamma & d_\gamma\esm=\sigma_\frak{a}^{-1}\gamma\sigma_\frak{a}$. Then
\begin{align*} 
\mathcal{S}_\frak{a}([\![\gamma]\!]) +\mathcal{S}_\frak{a}([\![\tau]\!]) -\mathcal{S}_\frak{a}([\![\gamma \tau]\!]) 
&= \frac{V}{4\pi}\left(\frac{c_{\gamma}}{c_{\gamma\tau}c_{\tau}}+\frac{c_{\tau}}{c_{\gamma}c_{\gamma\tau}} +\frac{c_{\gamma\tau}}{c_{\tau}c_{\gamma}}\right) \\
&\qquad - \phi\bpm a_\gamma & b_\gamma \\ c_\gamma & d_\gamma\epm  -\phi\bpm a_\tau & b_\tau \\ c_\tau & d_\tau\epm  + \phi\bpm a_{\gamma\tau} & b_{\gamma\tau} \\ c_{\gamma\tau} & d_{\gamma\tau}\epm \\
&= \frac{V}{4\pi}\left(\frac{c_{\gamma}}{c_{\gamma\tau}c_{\tau}}+\frac{c_{\tau}}{c_{\gamma}c_{\gamma\tau}} +\frac{c_{\gamma\tau}}{c_{\tau}c_{\gamma}}\right) \\
&\qquad + \omega(\gamma,\tau) +\frac14\left(\sign(c_{\gamma\tau})-\sign(c_\gamma)-\sign(c_\tau)\right)\\
&= \frac{V}{4\pi}\left(\frac{c_{\gamma}}{c_{\gamma\tau}c_{\tau}}+\frac{c_{\tau}}{c_{\gamma}c_{\gamma\tau}} +\frac{c_{\gamma\tau}}{c_{\tau}c_{\gamma}}\right) -\frac14\sign(c_\gamma c_\tau c_{\gamma\tau}),
\end{align*}
where the last step follows from \cite[Theorem 1.2]{me2} as cited above.

\begin{thm} 
Let $\Gamma\subset\SL_2(\R)$ be a non-uniform lattice with a cusp at $\infty$ and such that $\iota=\bsm &-1\\1&\esm\in\Gamma$. Then the reciprocity law
\begin{align*}
\mathcal{S}_\infty([\![\gamma]\!]) -\mathcal{S}_\infty([\![\gamma \iota]\!]) &= 
\mathcal{S}_r(c_\gamma,d_\gamma) -\mathcal{S}_r(d_\gamma, -c_\gamma) = \frac{V}{4\pi}\left(\frac{c_\gamma}{d_{\gamma}} + \frac{1}{c_\gamma d_{\gamma}} + \frac{d_{\gamma}}{c_\gamma}\right) - \frac14 \sign\left(c_\gamma d_{\gamma}\right)
\end{align*}
holds for any non-trivial double coset $[\![\gamma]\!]$ in $\Gamma_\infty\backslash\Gamma\slash\Gamma_\infty$ such that $[\![\gamma \iota]\!]$ is also non-trivial (i.e. $c_\gamma d_\gamma\neq0$).
\end{thm}
\begin{proof}
If we choose $\tau=\iota$ in the general formula above, then
$$
\mathcal{S}_\infty([\![\gamma]\!]) + \mathcal{S}_\infty([\![\iota]\!]) - \mathcal{S}_\infty([\![\gamma\iota]\!]) = \frac{V}{4\pi}\left(\frac{c_\gamma}{d_{\gamma}} + \frac{1}{c_\gamma d_{\gamma}} + \frac{d_{\gamma}}{c_\gamma}\right) - \frac14 \sign\left(c_\gamma d_{\gamma}\right).
$$
By definition, we have 
$$
\mathcal{S}_\infty([\![\iota]\!]) = \mathcal{S}_r(1,0) = - \phi(\iota) = -\psi(\iota) -\frac{1}{4}
$$
and, applying (\ref{cocycle}) repeatedly, 
$$
- 2\psi(\iota) =  \omega(\iota,\iota) -  \psi(-I) = \omega(\iota,\iota) + \frac12\omega(-I,-I)  - \frac12\psi(I) = \omega(\iota,\iota) + \frac12\omega(-I,-I)  +\frac12\omega(I,I).
$$
We can then read from \cite[Theorem 2.1]{me2} that $\psi(\iota)=-\frac14$.
\end{proof}

The reciprocity law, applied together with the periodicity relation $\mathcal{S}_r(c,d)=\mathcal{S}_r(c,d')$ for all $d\equiv d'$ mod $c$ terminating when $\mathcal{S}_r(1,0)=0$, goes to show that such Dedekind symbols are rational in the matrix coefficients of $\Gamma$. In the case of a Hecke triangle group $H_q$, which is generated by $\iota$ and $\tau_q=\bsm 1&\lambda_q\\&1\esm$, $\lambda_q=2\cos\frac{\pi}{q}$, the reciprocity law implies that up to a constant scaling by $\lambda_q$, the Dedekind symbol is rational. Alternatively, a simple cohomological argument leads to the same conclusion. Observe that the RHS of (\ref{cocycle}) defines an integral 2-cocycle on $\Gamma$. The rational class obtained under the natural map $H^2(\Gamma,\Z)\to H^2(\Gamma,\Q)$ is necessarily trivial, since $H^2(\Gamma,\Q)=0$. This means that there exists a function $f:\Gamma\to\Q$ such that $f(\gamma\tau)-f(\gamma)-f(\tau)=\psi(\gamma)-\psi(\gamma)-\psi(\tau)$. In the case $\Gamma=H_q$, the fact that $H_q=C_2\ast C_q$ implies that $H^1(H_q,\Q)=0$, and hence that the map $f$ is unique, i.e.~$f=\psi$.

\end{document}